\documentclass[12pt,reqno]{amsart}

\usepackage{amsthm, mathtools, amssymb, amsmath, graphicx, xcolor,nccmath}
\usepackage{enumerate}
\usepackage[margin=1in]{geometry}
\usepackage[doi=false,isbn=false,url=false,eprint=false,style=trad-plain]{biblatex}
\newtheorem{theorem}{Theorem}[section]
\newtheorem{lemma}[theorem]{Lemma}
\newtheorem{corollary}[theorem]{Corollary}
\theoremstyle{definition}

\newtheorem*{proposition*}{Proposition}
\newtheorem*{remark*}{Remark}

\mathtoolsset{showonlyrefs=true}

\numberwithin{equation}{section}

\newcommand{\R}{\mathbb{R}}
\newcommand{\Z}{\mathbb{Z}}
\newcommand{\Q}{\mathbb{Q}}
\newcommand{\C}{\mathbb{C}}

\newcommand{\vprod}[1]{\left\langle #1 \right\rangle}

\newcommand{\ptmatrix}[4]{\left( \begin{smallmatrix} #1 & #2 \\ #3 & #4 \end{smallmatrix} \right)}
\newcommand{\ptMatrix}[4]{\left( \begin{matrix} #1 & #2 \\ #3 & #4 \end{matrix} \right)}

\DeclareMathOperator{\SL}{SL}
\DeclareMathOperator{\Mp}{Mp}

\newcommand{\pfrac}[2]{\left(\frac {#1}{#2}\right)}
\renewcommand{\bar}{\overline}
\newcommand{\ep}{\varepsilon}

\newcommand{\e}{\mathfrak{e}}
\newcommand{\lf}{\left\lfloor\frac{\lambda}{2}\right\rfloor}

\newcommand{\ltf}{\left\lfloor\tfrac{\lambda}{2}\right\rfloor}

\begin{document}

\title[Weil bound for generalized Kloosterman sums]{The Weil bound for generalized Kloosterman sums of half-integral weight}

\author{Nickolas Andersen}
\email{nick@math.byu.edu}
\address{Mathematics Department, Brigham Young University, Provo, UT 84602}
\author{Gradin Anderson}
\email{gradinmanderson@gmail.com}
\address{Mathematics Department, Brigham Young University, Provo, UT 84602}
\author{Amy Woodall}
\email{amyew3@illinois.edu}
\address{Mathematics Department, University of Illinois at Urbana-Champaign, Urbana, IL 61801}

\begin{abstract}
Let $L$ be an even lattice of odd rank with discriminant group $L'/L$, and let $\alpha,\beta \in L'/L$.
We prove the Weil bound
for the Kloosterman sums $S_{\alpha,\beta}(m,n,c)$ of half-integral weight for the Weil Representation attached to $L$.
We obtain this bound by proving an identity that relates a divisor sum of Kloosterman sums to a sparse exponential sum.
This identity generalizes Kohnen's identity for plus space Kloosterman sums with the theta multiplier system.
\end{abstract}

\maketitle

\section{Introduction}

In 1926, Kloosterman \cite{kloosterman} introduced his eponymous exponential sum
\begin{equation} \label{eq:kloo-def-ordinary}
    S(m,n,c) = \sum_{d(c)^\times} e_c(m\bar d + n d), \qquad e_c(x) = e^{2\pi i x/c},
\end{equation}
in order to apply the circle method to the problem of representations of integers by quaternary quadratic forms.
Here the $\times$ superscript indicates that we sum over $d\in (\Z/c\Z)^\times$, and $\bar d$ denotes the inverse of $d$ modulo $c$.
Kloosterman proved that $S(m,n,p) \ll  p^{3/4}$ for any prime $p$.
This was subsequently improved by Weil in \cite{weil} to the sharp bound
\[
    |S(m,n,p)|\leq 2p^{\frac 12}.
\]
A theorem of Katz (see \cite{katz} and \cite{adolphson}) asserts that the Kloosterman angles $\theta_p(n)$ defined by the relation $S(1,n,p) = 2\sqrt p\cos \theta_p(n)$ are equidistributed with respect to the Sato--Tate measure.
For $c=p^\lambda$ with $\lambda \geq 2$, the Kloosterman sum can be evaluated explicitly (see Chapter~4 of \cite{iwaniec}) and for all $c>0$ we have
\[
    |S(m,n,c)| \leq \tau(c) (m,n,c)^{\frac 12} c^{\frac 12},
\]
where $\tau(c)$ is the number of divisors of $c$.
This inequality is called the Weil bound.

In this paper we prove a similar bound for Kloosterman sums of half-integral weight for the Weil representation,
which are defined precisely in Section~\ref{sec:background}.
Briefly, let $L$ be an even lattice with determinant $\Delta$ and discriminant group $L'/L$.
Then $|L'/L|=|\Delta|$.
For $\alpha,\beta\in L'/L$, let $\rho_{\alpha\beta}$ denote the coefficients of the Weil representation $\rho_L$.
These coefficients are given by an explicit exponential sum \eqref{eq:rho-alpha-beta-formula} involving values of the quadratic form $q:L'/L\to \Q/\Z$.
Suppose that $c \in \Z^+$, $\frac{m}{2\Delta}\in \Z+q(\alpha)$, and $\frac{n}{2\Delta}\in \Z+q(\beta)$, and let $k$ be a half-integer satisfying \eqref{eq:sigma-consistency}. 
We define the generalized Kloosterman sum as
\begin{equation} \label{eq:kloo-weil-def}
    S_{\alpha,\beta}(m,n,c) = e^{-\pi ik/2} \sum_{d(c)^\times} \bar \rho_{\alpha\beta}(\tilde\gamma) e_{2\Delta c}\left(ma+nd\right),
\end{equation}
where $\gamma=\ptmatrix abcd\in \SL_2(\Z)$ is any matrix with bottom row $(c\ d)$ and $\tilde\gamma = (\gamma,\sqrt{cz+d})$.
The coefficients $\rho_{\alpha\beta}$ satisfy $\rho_{\alpha\beta}(\tilde\gamma)\ll_L 1$, so the trivial bound $S_{\alpha,\beta}(m,n,c) \ll_L c$ holds.

Let $g$ denote the rank of $L$ and suppose that $g$ is odd.
In Section~\ref{sec:background} we show that
\[
    (-1)^{(g-1)/2} m\equiv 0,1\pmod{4}.
\]
In the following theorem, $\omega(c)$ denotes the number of distinct primes dividing $c$.

\begin{theorem} \label{thm:main-weil-bound}
    Suppose that $g$ is odd.
    Let $\alpha,\beta\in L'/L$ and let $m,n$ be integers satisfying $\tfrac{m}{2\Delta}\in \Z+q(\alpha)$ and $\tfrac{n}{2\Delta}\in \Z+q(\beta)$.
    Write $m=m_0v^2$ where $(-1)^{(g-1)/2}m_0$ is a fundamental discriminant.
    If $(v,\Delta)=1$ then
    \begin{equation} \label{eq:main-weil-bound}
        S_{\alpha,\beta}(m,n,c) \ll_L 2^{\omega(c)} \tau((v,c))  (m_0n,c)^{\frac 12} c^{\frac 12}.
    \end{equation}
\end{theorem}

\begin{remark*}
    The implied constant is of the form $|\Delta|^{Ag}$ for some absolute constant $A$, which can be computed explicitly using the results of Section~\ref{sec:counting}.
\end{remark*}

\begin{remark*}
    It is likely that a bound similar to \eqref{eq:main-weil-bound} holds when $g$ is even, but the methods of this paper are not suited to that case.
\end{remark*}

We will deduce \eqref{eq:main-weil-bound} from an identity that relates a sum of Kloosterman sums to a sparse exponential sum.
In Section~\ref{sec:exp-sum-identity} we state this identity and use it to prove Theorem~\ref{thm:main-weil-bound}.
In Section~\ref{sec:background} we provide background on the discriminant group of a lattice, the Weil representation, and Gauss sums.
The proof of the exponential sum identity is Section~\ref{sec:proof-identity}.

\section{The exopnential sum identity} \label{sec:exp-sum-identity}

We begin by discussing Kohnen's identity \cite[Proposition~5]{kohnen} for plus space Kloosterman sums with the theta multiplier system and a similar identity for the eta multiplier system proved by the first author \cite[Proposition~6]{andersen}.
Let $k\in \R$ and let $\nu$ be a multiplier system (see \cite[Section~2.6]{iwaniec}) of weight $k$ on a congruence subgroup $\Gamma\subseteq \SL_2(\Z)$ containing $T=\ptmatrix 1101$.
If $\Gamma_\infty=\langle \pm T \rangle$ denotes the stabilizer of $\infty$ in $\Gamma$, define the Kloosterman sum with multiplier system $\nu$ by
\begin{equation}
    S(m,n,c,\nu) = \sum_{\gamma = \ptmatrix{a}{b}{c}{d}\in \Gamma_\infty \backslash \Gamma / \Gamma_\infty} \bar \nu(\gamma) e_c(ma+nd).
\end{equation}
Here the sum is only well-defined if $m,n\in \Q$ satisfy a consistency condition involving the value of $\nu(T)$.
These Kloosterman sums appear in the Fourier coefficients of Poincar\'e series of weight $k$ with multiplier system $\nu$.
See \cite{selberg} for more details.

Let $\nu_\theta$ denote the multiplier system of weight $\frac 12$ on $\Gamma_0(4)$ for the theta function $\theta(z) = \sum_{n\in \Z}e(n^2z)$.
Kohnen identified a distinguished subspace of modular forms for $\nu_\theta$ which he called the plus space.
Projection of Poincar\'e series to the plus space naturally introduces the modified Kloosterman sums
\begin{equation}
    S^+(m,n,4c,\nu_\theta) = (1-i) S(m,n,4c,\nu_\theta) \times
    \begin{cases}
        1 & \text{ if } c \text{ is even}, \\
        2 & \text{ if } c \text{ is odd}.
    \end{cases}
\end{equation}
Kohnen's identity relates these plus space Kloosterman sums to a sparse quadratic Weyl sum.
To precisely state the identity, we first fix some notation.
Suppose that $N$ is a positive integer and that $m$ and $n$ are squares modulo $4N$.
Further suppose that $m$ is a fundamental discriminant. 
Then $m\equiv 0,1\pmod{4}$ and either $m$ is odd and squarefree, or $\frac m4 \equiv 2,3\pmod{4}$ and $\frac m4$ is squarefree.
Following \cite[Section I.2]{GKZ}, if $b^2-4Nac=mn$, define
\begin{equation}
    \chi_m(aN,b,c) = 
    \begin{cases}
        \pfrac mr & \text{ if } (a,b,c,m)=1, \\
        0 & \text{ otherwise},
    \end{cases}
\end{equation}
where, in the first case, $r$ is any integer coprime to $m$ represented by the quadratic form $aN_1x^2+bxy+cN_2y^2$, for some splitting $N=N_1N_2$, $N_1>0$.
Proposition~1 of \cite{GKZ} gives several properties of $\chi_m(aN,b,c)$, including that it is well-defined and independent of choice of splitting $N=N_1N_2$.

\begin{proposition*}[Kohnen, Proposition~5 of \cite{kohnen}]
    Suppose that $m,n\equiv 0,1\pmod{4}$ and that $m$ is a fundamental discriminant.
    Then for all $v\in \Z$ we have
\begin{equation} \label{eq:kohnen-id}
    \sum_{u\mid (v,c)} \pfrac mu \sqrt{\frac uc} \, S^+\!\left(\mfrac{mv^2}{u^2},n,\mfrac{4c}u,\nu_\theta\right) = 4\sum_{\substack {\ell(2c) \\ \ell^2 \equiv mn(4c)}} \chi_m\left(c,\ell,\mfrac{\ell^2-mn}{4c}\right)e_{2c}(\ell v).
\end{equation}
\end{proposition*}

The sum on the right-hand side of \eqref{eq:kohnen-id} is quite small in absolute value; in particular, if $(mn,4c)=1$, the sum is $\ll c^\ep$.
By Mobi\"us inversion in two variables (see Corollary~\ref{cor:mobius-inv} below) it follows that $S^+(m,n,4c,\nu_\theta) \ll c^{1/2+\ep}$.

\begin{remark*}
    One notable use of
    Kohnen's identity is by Duke, Imamo\=glu, and T\'oth  \cite{DIT1,DIT2} as a bridge connecting coefficients of  half-integral weight forms and cycle integrals of weight zero forms.
\end{remark*}

In \cite{andersen} the first author proved an analogue of Kohnen's identity for Kloosterman sums with the Dedekind eta multiplier system $\nu_\eta$ of weight $\frac 12$ on $\SL_2(\Z)$.
Up to a constant, $S(\frac{1}{24},\frac{1}{24}-n,c,\nu_\eta)$ equals the sum $A_c(n)$ appearing in the Hardy-Ramanujan-Rademacher formula for the partition function $p(n)$ (see \cite{rademacher,AA}).

\begin{proposition*}(Andersen, Proposition~6 of \cite{andersen}) 
    Suppose that $m,n\equiv 1\pmod{24}$ and that $m$ is a fundamental discriminant. Then for all $v\in \Z$ with $(v,6)=1$ we have
    \begin{multline}
    2\sqrt{-3i}\sum_{u\mid (v,c)} \pfrac{12}{v/u} \pfrac{m}{u} \sqrt{\frac uc} \, S\left(\mfrac{mv^2}{24u^2}, \mfrac{n}{24}, \mfrac{c}{u},\nu_\eta\right) \\ = \sum_{\substack{\ell(12c) \\ \ell^2\equiv mn(24c)}} \pfrac{12}{\ell} \chi_m\left(6c,\ell,\mfrac{\ell^2-mn}{24c}\right)  e_{12c}(\ell v).
    \end{multline}
\end{proposition*}

The Kloosterman sums appearing in these identities can be written as linear combinations of the sums \eqref{eq:kloo-weil-def}.
We indicate how this is done for $\nu_\eta$; the construction is similar for $\nu_\theta$.
Let $L$ denote the lattice $\Z$ with bilinear form $\vprod {x,y} = 12xy$ (use $\vprod{x,y}=2xy$ instead for $\nu_\theta$).
The dual lattice is $L'=\frac{1}{12}\Z$, so for $\alpha,\beta\in L'/L$ we can write $\alpha = \frac{h}{12}$ and $\beta=\frac{j}{12}$ for $h,j\in \Z/12\Z$.
In Section~\ref{sec:background} we will prove that
for any $h\in \Z/12\Z$ with $(h,6)=1$ we have
\begin{equation} \label{eq:eta-weil-connection}
    S\left(\tfrac{m}{24},\tfrac{n}{24},c,\nu_\eta\right) = \pfrac{12}{h}\sum_{j(12)} \pfrac{12}{j} S_{\alpha,\beta}(m,n,c).
\end{equation}

Our first version of the exponential sum identity is a direct generalization of Kohnen's identity for even lattices of rank $1$.
In this case, without loss of generality we can take $L=\Z$ and $\vprod{x,y}=\Delta xy$, where $\Delta=\pm 2N$ for some $N\in \Z^+$.
Then $q(x) := \frac 12 \vprod{x,x} = \frac 12 \Delta x^2$ and we can write $\alpha,\beta\in L'/L$ as $\alpha = \frac{a}{\Delta}$ and $\beta = \frac{b}{\Delta}$ for $a,b\in \Z/2N\Z$.
Define $\sigma$ by \eqref{eq:sigma-consistency}.

\begin{theorem} \label{thm:exp-sum-identity-g=1}
    Suppose that $L$ has rank $1$.
Let $\alpha=\frac{a}{\Delta},\beta=\frac{b}{\Delta}\in L'/L$, and let $m,n$ be integers satisfying $m\equiv a^2\pmod{4N}$ and $n \equiv b^2 \pmod{4N}$.
Suppose that $m$ is a fundamental discriminant.
Then for any $v\in \Z$ and any $c\geq 1$ we have
\begin{equation}
    \sum_{u\mid (v,c)} \pfrac{m}{u} \sqrt{\frac uc} \, S_{\alpha \frac vu, \beta}\left( mv^2/u^2, n, \mfrac cu \right) \\
    = \frac{i^{-\sigma}}{\sqrt{2N}}\sum_{\substack{\ell(2Nc) \\ \ell \equiv ab(2N) \\ \ell^2\equiv mn(4Nc) }} 
    \chi_{m}\left(Nc,\ell,\mfrac{\ell^2-mn}{4Nc}\right) e_{\Delta c}(\ell v).
    \label{eq:main-identity}
\end{equation}
\end{theorem}

By \cite[Proposition~1]{GKZ} the character $\chi_m$ can be computed using the formula
\begin{equation}
    \chi_m \left(Nc, \ell, \mfrac{\ell^2-mn}{4Nc}\right) \\= \prod_{\substack{p^\lambda\parallel c \\ p\nmid m}} \pfrac{m}{p^{\lambda}}
    \prod_{\substack{p^\lambda\parallel c \\ p\mid m}} \pfrac{m/p^*}{p^{\lambda+\nu}} \pfrac{p^*}{(\ell^2-mn)/p^{\lambda+\nu}}
\end{equation}
where $p^\nu\parallel 2\Delta$ and
\[
    p^* = 
    \begin{cases}
        (\frac{-1}{p})p & \text{ if $p$ is odd}, \\
        (\frac{-1}{m'})2^\mu & \text{ if $p=2$ and $m=2^\mu m'$ with $m'$ odd}.
    \end{cases}
\]

Our second version of the exponential sum identity holds for any lattice of odd rank at the cost of having less-precise information at the ``bad'' primes, i.e.~primes dividing $2(m,\Delta,c)$.
At these primes we will need to count the number of solutions to the quadratic congruence
\begin{equation}
    \tilde m x^2 - \vprod{\alpha,y}x - q(y) + \vprod{\beta,y} - \tilde \ell x + \tilde n \equiv 0\pmod{p^j},
\end{equation}
where $x\in \Z/p^j\Z$, $y\in L/p^j L$, and $\tilde m = \frac{m}{2\Delta}-q(\alpha)$, $\tilde \ell = \frac{\ell}{\Delta}-\vprod{\alpha,\beta}$, and $\tilde n = \frac{n}{2\Delta}-q(\beta)$.
(The quantities $\tilde m$, $\tilde \ell$, and $\tilde n$ are integers in each context in which this congruence appears.)
Let $N(p^j)$ denote the number of such solutions.
We define a function $\xi_{\alpha,\beta}(\ell,m,n,c)$ at prime powers $c=p^\lambda$ and extend to all $c$ multiplicatively.
Write $|\Delta|=2N$ and
\[
    m_L = (-4)^{(g-1)/2}m.
\]
(In Section~\ref{sec:background} we will show that $\Delta$ is even whenever $g$ is odd.)
In particular, note that $m_L=m$ when $g=1$.
\begin{enumerate}
    \item If $p$ is odd and $(m,\Delta,p)=1$ then
    \begin{equation}
        \xi_{\alpha,\beta}(\ell,m,n,p^\lambda) = 
        \begin{dcases}
            \pfrac{m_L}{p^\lambda}  & \text{ if } p\nmid m \text{ and } \ell^2\equiv mn\pmod{p^{\lambda+\nu}}, \\
            \pfrac{m_L/p^*}{p^{\lambda}} \pfrac{p^*}{(\ell^2-mn)/p^{\lambda}}  & \text{ if } p \mid m \text{ and } \ell^2\equiv mn\pmod{p^{\lambda}}, \\
            0 & \text{ otherwise}.
        \end{dcases}
    \end{equation}
    (Note that in the second case $\nu=0$.)
    \item If $p=2$ or if $p$ is odd and $(m,\Delta,p)>1$ then
    \begin{equation}
        \xi_{\alpha,\beta}(\ell,m,n,p^\lambda) =
        \begin{dcases}
        p^{-\lambda \frac{g+1}{2}} \left(N(p^\lambda) - p^gN(p^{\lambda-1})\right) & \text{ if }\ell^2 \equiv mn \pmod{2Np^{2\lf}}, \\
        0 & \text{ otherwise}.
        \end{dcases}
    \end{equation}
\end{enumerate}
Here $p^\nu \parallel 2\Delta$, as in the definition of $\chi_m$.

\begin{theorem} \label{thm:exp-sum-general}
    Suppose that $g:=\operatorname{rank} L>1$ is odd. 
    Let $\alpha,\beta \in L'/L$ and let $m,n$ be integers satisfying $\frac{m}{2\Delta}\in \Z+q(\alpha)$ and $\frac{n}{2\Delta} \in \Z+q(\beta)$.
    Suppose that $(-1)^{(g-1)/2}m$ is a fundamental discriminant. 
    Then for any $v\in \Z$ and any $c\geq 1$ we have
\begin{equation} \label{eq:exp-sum-general}
    \sum_{u\mid (v,c)} \pfrac{m_L}{u} \sqrt{\frac uc} \, S_{\alpha \frac vu, \beta}\left( mv^2/u^2, n, \mfrac cu \right) \\
    = \frac{i^{-\sigma}}{\sqrt{2N}} \sum_{\substack{\ell(2Nc) \\ \frac{\ell}{\Delta} \equiv \vprod{\alpha,\beta}(1)}} 
    \xi_{\alpha,\beta}(\ell,m,n,c) e_{\Delta c}(\ell v).
\end{equation}
Furthermore, for all $\ell,m,n,c$ we have
\begin{equation} \label{eq:xi-L-bound}
    \xi_{\alpha,\beta}(\ell,m,n,c)\ll_L 1.
\end{equation}
\end{theorem}

\begin{remark*}
    The Kloosterman sums $S_{\alpha,\beta}(0,n,c)$ appear naturally in the Fourier coefficients of Eisenstein series for the Weil representation, which are studied in \cite{bruinier-kuss} and \cite{schwagenscheidt}.
    In the formulas given in those papers, quantities analogous to $N(p^\lambda)-p^g N(p^{\lambda-1})$ also appear at the bad primes. 
\end{remark*}

\begin{corollary} \label{cor:mobius-inv}
    With the assumptions of Theorem~\ref{thm:exp-sum-general}, we have
    \begin{equation}
        S_{\alpha v,\beta}\left(mv^2,n, c\right) \\=  \frac{i^{-\sigma}\sqrt c}{\sqrt{2N}} \sum_{u\mid(v,c)} \mu(u) \pfrac{m_L}{u}\sum_{\substack{\ell(2Nc/u) \\ \frac{\ell}{\Delta} \equiv \vprod{\alpha,\beta}(1)}} \xi_{\alpha,\beta}(\ell,m,n,c/u) e_{\Delta c}(\ell v).
    \end{equation}
\end{corollary}

\begin{proof}
    We apply M\"obius inversion in two variables.
    The identity \eqref{eq:exp-sum-general} can be written
    \[
        \sum_{u\mid (v,c)} \pfrac{m_L}{u} f(v/u,c/u) = g(v,c).
    \]
    Therefore
    \begin{align}
        \sum_{u\mid (v,c)} \mu(u) \pfrac{m_L}{u} g(v/u,c/u) 
        &= \sum_{u\mid (v,c)} \mu(u) \pfrac{m_L}{u} \sum_{w\mid (v/u,c/u)} \pfrac{m_L}{w} f(v/uw,c/uw) \\
        &= \sum_{\substack{v=uwa \\ c=uwb}} \mu(u) \pfrac{m_L}{uw} f(a,b) \\
        &= \sum_{\substack{a\mid v, \, b\mid c \\ v/a=c/b}} \pfrac{m_L}{v/a} f(a,b) \sum_{u\mid v/a} \mu(u) = f(v,c).
    \end{align}
    The corollary follows immediately.
\end{proof}

We can now prove Theorem~\ref{thm:main-weil-bound}, assuming the truth of Theorems~\ref{thm:exp-sum-identity-g=1} and \ref{thm:exp-sum-general}.

\begin{proof}[Proof of Theorem~\ref{thm:main-weil-bound}]
    Suppose that $\alpha,\beta\in L'/L$ and that $\frac{m}{2\Delta}-q(\alpha),\frac{n}{2\Delta}-q(\beta)\in \Z$. Write $m=m_0v^2$ with $(-1)^{(g-1)/2}m_0$ fundamental and $(v,\Delta)=1$.
    We assume here that $g>1$; the case $g=1$ is similar and a bit easier.
    Since $v$ and $|L'/L|$ are coprime, there exists an $\alpha'\in L'/L$ such that $\alpha=v\alpha'$.
    Since $m_0v^2-\vprod{\Delta v\alpha',v\alpha'}\in 4N\Z$, we have $\frac{m_0}{2\Delta}-q(\alpha')\in \Z$.
    By Corollary~\ref{cor:mobius-inv} we have
    \begin{equation}
        S_{\alpha' v,\beta}\left(m_0v^2,n, c\right) \ll_L \sqrt c  \sum_{u\mid (v,c)} R(m_0n,c/u),
    \end{equation}
    where $R(y,c)$ is the number of solutions to the quadratic congruence $x^2\equiv y\pmod{c}$.
    Since $R(y,c)$ is multiplicative as a function of $c$, it suffices to evaluate $R(y,p^\lambda)$ for each prime power $p^\lambda\parallel c$.

    Suppose that $p$ is odd.
    (In the case $p=2$ the estimates given below are correct if we multiply each upper bound by $2$.)
    If $p\nmid y$ then $R(y,p^\lambda)\leq 2$ by a simple argument using Hensel's lemma.
    Now suppose that $y=p^\mu y'$ with $p\nmid y'$ and $\mu\geq 1$.
    Then any solution to $x^2\equiv y\pmod{p^\lambda}$ can be written $x=p^\delta x'$, where $\delta = \min(\lceil \frac\mu 2 \rceil, \lceil \frac\lambda 2 \rceil)$.
    Then
    \begin{equation}
        R(y,p^\lambda) \leq 
        \begin{cases}
            2p^{\mu-\delta} & \text{ if } \mu<\lambda, \\
            p^{\lambda-\delta} & \text{ if } \mu\geq \lambda.
        \end{cases}
    \end{equation}
    It follows that $R(y,p^\lambda) \leq 2p^{\min(\mu,\lambda)/2}$, so
    \[
        R(y,c) \leq 2^{\omega(c)+1} (y,c)^{\frac 12}.
    \]
    Theorem~\ref{thm:main-weil-bound} follows.
\end{proof}

\section{Background} \label{sec:background}

\subsection{Lattices and discriminant groups}

Let $L$ be an even lattice with nondegenerate symmetric bilinear form $\langle \cdot, \cdot \rangle$, and let $q(x) = \frac 12 \vprod {x,x}$ denote the associated $\Z$-valued quadratic form.
Let $L'$ denote the dual lattice
\begin{equation} \label{eq:L'-def}
    L' = \left\{ x\in L\otimes \Q : \vprod {x,y} \in \Z \text{ for all }y\in L \right\};
\end{equation}
then the quotient $L'/L$ is a finite abelian group.
We denote the standard basis of $\C[L'/L]$ by $\{\mathfrak e_\alpha : \alpha \in L'/L\}$.

By identifying $L$ with $\Z^g$ we may write $\vprod{x,y} = x^TMy$ for all $x,y\in \Z^g$, for some symmetric integer matrix $M$ with even diagonal.
Let $\Delta = \det M$; then $|L'/L|=|\Delta|$.
If $\alpha\in L'$ then we can write $\alpha = M^{-1}a$ for some $a\in \Z^g$ and we have $\Delta \alpha\in L$.

Here we give a few lemmas that will be useful in the following section.

\begin{lemma} \label{lem:alpha-beta-gamma}
    For all $\alpha,\beta,\gamma\in L'$ we have
    \[\Delta(\vprod{\gamma,\alpha}\beta-\vprod{\gamma,\beta}{\alpha})\in L.\]
\end{lemma}
\begin{proof}
    We write $\alpha=M^{-1}a,\beta=M^{-1}b,\gamma=M^{-1}c$ for some $a,b,c\in \Z^g$.
    In this notation, it suffices to prove that the vector
    \[x = \det(M)(M^{-1}bc^TM^{-1}a-M^{-1}ac^TM^{-1}b)\]
    is in $\Z^g$.
    Notice that this quantity is linear in $a,b,$ and $c$, so we may assume that $a=e_i$, $b=e_j$, and $c=e_k$, where $e_i$ is the $i$-th standard basis vector. 
    If $x_\ell$ denotes the $\ell$-th component of $x$, then
    \[\frac{x_\ell}{\det (M)} = e_\ell^T M^{-1}e_je_k^TM^{-1}e_i-e_\ell^TM^{-1}e_ie_k^TM^{-1}e_j.\] 
    Note that $\det(M)e_\ell^TM^{-1}e_j=(-1)^{j+\ell}M_{j,\ell}$, where $M_{j,\ell}$ denotes the $j,\ell$-th minor of $M$, and similarly for the other products, so we obtain 
    \[\det(M)x_\ell = (-1)^{i+j+k+\ell} \left(M_{j,\ell}M_{i,k}-M_{i,\ell}M_{j,k}\right).\] 
    Sylvester's determinant identity \cite{bareiss}, also known as the Desnanot-Jacobi identity, shows that the expression on the right-hand side is divisible by $\det M$.
    It follows that $x_\ell \in \Z$.
\end{proof}

If $g$ is odd then $\Delta$ is even by Lemma~14.3.21 of \cite{cohen-stromberg}.
As in the introduction, we write $|\Delta| = 2N$ with $N\in \Z^+$.

\begin{lemma} \label{lem:l^2-mn}
    Suppose that $g$ is odd.
    Let $\alpha,\beta\in L'/L$.
    Let $\ell,m,n\in \Z$ such that
    \[
        \tfrac{\ell}{\Delta}-\vprod{\alpha,\beta}\in \Z, \quad \tfrac{m}{2\Delta}-q(\alpha)\in \Z, \ \text{ and } \ \tfrac{n}{2\Delta}-q(\beta) \in \Z.
    \]
    Then $\ell^2\equiv mn \pmod{2N}$.
    If $g=1$ then $\ell^2\equiv mn\pmod{4N}$.
\end{lemma}

\begin{proof}
    The assumptions on $\ell$, $m$, and $n$ are equivalent to
    \begin{align}
        \ell &\equiv \Delta\vprod{\alpha,\beta} \pmod{2N}, \\
        m &\equiv \Delta\vprod{\alpha,\alpha} \pmod{4N}, \\
        n &\equiv \Delta\vprod{\beta,\beta} \pmod{4N}.
    \end{align}
    It follows that
    \[
        \ell^2-mn \equiv \Delta^2(\vprod{\alpha,\beta}^2-\vprod{\alpha,\alpha}\vprod{\beta,\beta}) \pmod{4N}.
    \]
    If $g=1$ then $\vprod{\alpha,\beta}^2=\vprod{\alpha,\alpha}\vprod{\beta,\beta}$ so $\ell^2\equiv mn\pmod{4N}$.
    If $g>1$ then the lemma will follow if we can show that $\Delta(\vprod{\alpha,\beta}^2-\vprod{\alpha,\alpha}\vprod{\beta,\beta})$ is an integer.
    By Lemma~\ref{lem:alpha-beta-gamma} we have
    \[
        x := \Delta (\vprod{\alpha,\beta}\alpha - \vprod{\alpha,\alpha}\beta) \in L.
    \]
    Thus $\vprod{x,\beta}\in \Z$, which completes the proof.
\end{proof}

\begin{lemma}
    Suppose that $g$ is odd.
    Let $\alpha \in L'/L$ and suppose that $\frac{m}{2\Delta}-q(\alpha)\in \Z$.
    Then 
    \[
        (-1)^{(g-1)/2} m\equiv 0,1\pmod{4}.
    \]
\end{lemma}
\begin{proof}
    Let $\alpha \in L'/L$ and write $\alpha = M^{-1}a$ for $a\in \Z^g$ so that $q(\alpha)=\frac 12 a^TM^{-1}a$.
    Then $-m=\frac m\Delta \det(-M)=\det(S)$, where $S$ is the block matrix 
    \begin{equation}
        S=\ptMatrix{\frac{m}{\Delta}-a^TM^{-1}a}{-a^T}{-a}{-M}.
    \end{equation}
    Note that $S$ is a symmetric integer matrix with even diagonal, so the result follows from Lemma~14.3.20 of \cite{cohen-stromberg}.
\end{proof}

\subsection{The Weil representation}

Good references for the background material in this subsection are \cite[Chapter~1]{bruinier} and \cite[Chapter~14]{cohen-stromberg}.
Let $\Mp_2(\R)$ be the metaplectic group, the elements of which are of the form $(\gamma,\phi)$, where $\gamma=\ptmatrix abcd\in \SL_2(\R)$ and $\phi:\mathbb H\to\C$ is a holomorphic function with $\phi^2(\tau)=c\tau+d$.
The group law on $\Mp_2(\R)$ is given by
\begin{equation}
	(\gamma_1,\phi_1(\tau))(\gamma_2,\phi_2(\tau)) = (\gamma_1\gamma_2, \phi_1(\gamma_2\tau)\phi_2(\tau)).
\end{equation}
Let $\Mp_2(\Z)$ denote the inverse image of $\SL_2(\Z)$ under the covering map $(\gamma,\phi)\mapsto \gamma$.
Then $\Mp_2(\Z)$ is generated by the elements
\begin{equation}
	T=\left(\ptmatrix 1101, 1\right) \quad \text{ and } \quad S = \left(\ptmatrix 0{-1}10,\sqrt\tau\right)
\end{equation}
and the center of $\Mp_2(\Z)$ is generated by
\begin{equation}
	Z = S^2 = (ST)^3 = \left(\ptmatrix {-1}00{-1},i\right).
\end{equation}
The Weil representation associated with the lattice $L$ is the unitary representation
\[
    \rho_L : \Mp_2(\Z) \to \C[L'/L]
\]
given by
\begin{align}
    \rho_L(T) \e_\alpha &= e(q(\alpha))\e_\alpha, \label{eq:T-transform} \\
    \rho_L(S) \e_\alpha &= \frac{i^{(b^--b^+)/2}}{\sqrt{|L'/L|}} \sum_{\beta \in L'/L} e(-\vprod{\alpha,\beta})\e_\beta.
\end{align}
Here $(b^+,b^-)$ is the signature of $L$.
For $\frak g\in \Mp_2(\Z)$ we define the coefficient $\rho_{\alpha\beta}(\frak g)$ of the representation $\rho_L$ by
\begin{equation}
    \rho_L(\frak g) \frak e_\beta = \sum_{\alpha\in L'/L} \rho_{\alpha\beta}(\frak g) \frak e_\alpha.
\end{equation}
Shintani \cite{shintani} gave the following formula for the coefficients $\rho_{\alpha\beta}(\frak g)$: if $\frak g = (\ptmatrix abcd, \sqrt{cz+d})$ and $c>0$, then
\begin{equation} \label{eq:rho-alpha-beta-formula}
    \rho_{\alpha\beta}(\frak g) = \frac{i^{\frac{b^--b^+}{2}}}{c^{(b^++b^-)/2}\sqrt{|L'/L|}} \sum_{r\in L/cL} e_c(aq(\alpha+r)-\vprod{\beta,\alpha+r}+dq(\beta)).
\end{equation}
Since $\rho_L$ factors through a double cover of the finite group $\SL_2(\Z/4N\Z)$, we have the upper bound $\rho_{\alpha\beta}(\mathfrak g)\ll_L 1$.

If $f:\mathbb H\to \C[L'/L]$ is a modular form for the Weil representation then $f$ satisfies the transformation law
\[
    f(\gamma z) = \phi^{2k}(z) \rho_L(\mathfrak g) f(z) \quad \text{ for all } \mathfrak g = (\gamma,\phi)\in \Mp_2(\Z).
\]
Setting $\mathfrak g=Z$, we find that such an $f$ satisfies $f = (-1)^{2k+b^--b^+} f$.
Thus $f=0$ unless $k$ satisfies the consistency condition
\begin{equation} \label{eq:sigma-consistency}
    \sigma := k + \tfrac 12(b^--b^+) \in \Z.
\end{equation}

Suppose that $k$ satisfies \eqref{eq:sigma-consistency}.
Then for $c \in \Z^+$, $\frac{m}{2\Delta}\in \Z+q(\alpha)$, and $\frac{n}{2\Delta}\in \Z+q(\beta)$, 
we define the generalized Kloosterman sum as
\begin{equation} \label{eq:kloo-weil-def-2}
    S_{\alpha,\beta}(m,n,c) = e^{-\pi ik/2} \sum_{d(c)^\times} \bar \rho_{\alpha\beta}(\tilde\gamma) e_{2\Delta c}\left(ma+nd\right).
\end{equation}
Here $\gamma=\ptmatrix abcd\in \SL_2(\Z)$ is any matrix with bottom row $(c\ d)$ and $\tilde\gamma = (\gamma,\sqrt{cz+d})$ is a lift of $\gamma$ to $\Mp_2(\Z)$.
By \eqref{eq:T-transform} we have $\rho_{\alpha\beta}(T^r \frak g T^s) = e(rq(\alpha)+sq(\beta))\rho_{\alpha\beta}(\mathfrak g)$, so the sum \eqref{eq:kloo-weil-def-2} is independent of the choice of representatives for $(\Z/c\Z)^\times$ and the choice of matrix $\gamma$.

\begin{remark*}
While the weight $k$ does not play a major role in the definition \eqref{eq:kloo-weil-def-2}, we refer to the sums as half-integral weight Kloosterman sums when $g=b^++b^-$ is odd because of condition \eqref{eq:sigma-consistency}.
\end{remark*}

We conclude this subsection by proving equation \eqref{eq:eta-weil-connection}.
Let $L$ denote the lattice $\Z$ with bilinear form $\vprod {x,y} = 12xy$ (use $\vprod{x,y}=2xy$ instead for $\nu_\theta$).
The dual lattice is $L'=\frac{1}{12}\Z$, so we can write $\alpha = \frac{h}{12}$ and $\beta=\frac{j}{12}$ for $h,j\in \Z/12\Z$.
If $F(z) = \sum_{h(12)} \pfrac{12}{h} \eta(z) \e_{\alpha}$ (where we emphasize that $\e_\alpha$ depends on $h$) then
\begin{equation}
    F(\gamma z) = \rho_L(\tilde\gamma) (cz+d)^{\frac 12} F(z) \qquad \text{ for all }\gamma = \ptmatrix abcd\in \SL_2(\Z),
\end{equation}
where $\tilde \gamma=(\gamma,\sqrt{cz+d})\in \Mp_2(\Z)$ (see \cite[Section~3.2]{BO}). 
Therefore
\begin{align*}
    \nu_\eta(\gamma) F(z) &= (cz+d)^{-\frac 12}F(\gamma z) \\ &= \sum_{h(12)} \pfrac{12}{h} \eta(z) \rho_L(\tilde \gamma)\e_\alpha = \sum_{j(12)} \sum_{h(12)} \pfrac{12}{h}\rho_{\beta\alpha}(\tilde\gamma) \eta(z) \e_\beta,
\end{align*}
from which it follows that 
\begin{equation}
    \nu_\eta(\gamma) = \pfrac{12}{h} \sum_{j(12)} \pfrac{12}{j} \rho_{\alpha\beta}(\tilde\gamma) \qquad \text{ for all }\gamma \in \SL_2(\Z),
\end{equation}
for any $h\in \Z/12\Z$ with $(h,6)=1$.
Thus, for all such $h$ we have
\begin{equation}
    S\left(\tfrac{m}{24},\tfrac{n}{24},c,\nu_\eta\right) = \pfrac{12}{h}\sum_{j(12)} \pfrac{12}{j} S_{\alpha,\beta}(m,n,c).
\end{equation}

\subsection{Gauss sums}

Let $G(c)$ denote the Gauss sum
\begin{equation}
    G(c) = \sum_{x(c)} e_c(x^2).
\end{equation}
The evaluation of these sums is a classical result; see Chapter~1 of \cite{BEW} for a thorough treatment.
For odd $c$ we have
\begin{equation} \label{eq:gauss-odd}
    G(c) = \ep_c \sqrt{c},
\end{equation}
where
\begin{equation}
    \ep_c = 
    \begin{cases}
        1 & \text{ if }c\equiv 1\pmod{4}, \\
        i & \text{ if }c\equiv 3\pmod{4}.
    \end{cases}
\end{equation}
Furthermore, if $(a,c)=1$ then
\begin{equation} \label{eq:Gauss-y}
    \sum_{x(c)} e_c(ax^2) = \pfrac{a}{c} G(c).
\end{equation}
When $c=2^\lambda$ we will encounter the more general Gauss sums
\begin{equation}
    G(a,b,c) = \sum_{x(c)} e_{c}(ax^2+bx).
\end{equation}
These are evaluated in Chapter~1 of \cite{BEW}.
We have $G(a,b,c) = 0$ unless $(a,c)\mid b$, and if $a$ is odd then
\begin{equation} \label{eq:gauss-even}
    \sum_{x(2^\lambda)} e_{2^\lambda}(ax^2+bx) = 
    \begin{cases}
        2 & \text{ if } \lambda = 1 \text{ and $b$ is odd}, \\
        e_{2^\lambda}\left(-\bar a (b/2)^2\right)(1+i)\ep_a^{-1} \pfrac{2}{a}^\lambda 2^{\lambda/2} & \text{ if }\lambda\geq 2 \text{ and $b$ is even}, \\
        0 & \text{ otherwise},
    \end{cases}
\end{equation}
where $\bar a a \equiv 1\pmod{2^\lambda}$.

More generally, suppose that $f$ is a quadratic form on $\Z^g$, given by $f(x) = \frac 12 x^T M x$, where $M$ is a symmetric $g\times g$ integer matrix with even  diagonal.
Let $\Delta = \det M$.
If $c$ is odd and $(\Delta,c)=1$ then
\begin{equation} \label{eq:Gauss-sum-quadratic-form}
    \sum_{x \in (\Z/c\Z)^g} e_c(f(x)) = \pfrac{\bar 2^g\Delta}{c} G(c)^{g},
\end{equation}
where $\bar 2 2\equiv 1\pmod{c}$.
This formula is proved by Weber in \cite[Section~6]{weber}, see also \cite{cohen}.
It can be proved by first reducing to the case where $c=p^\lambda$ is a prime power, then using the fact that $f$ can be diagonalized over $\Z_p$ when $p$ is odd.

Lastly, we will encounter the sum
\begin{equation}
    T(n,p^\lambda) = \sum_{d(p^\lambda)^\times} \pfrac{d}{p} e_{p^\lambda}(dn)
\end{equation}
for an odd prime $p$.
By replacing $d$ by $d+p$, we see that $T(n,p^\lambda) = 0$ unless $n\equiv 0\pmod{p^{\lambda-1}}$.
In that case, $T(n,p^\lambda) = p^{\lambda-1}T(n/p^{\lambda-1},p)$, and this latter sum is the Gauss sum attached to the character $(\frac{\cdot}{p})$, which is evaluated in \cite[Chapter~1]{BEW}.
We conclude that
\begin{equation} \label{eq:twisted-exp-sum}
    \sum_{d(p^\lambda)^\times} \pfrac{d}{p} e_{p^\lambda}(d n) = 
    \begin{dcases}
        \varepsilon_p p^{\lambda-\frac12} \pfrac{n/p^{\lambda-1}}{p} & \text{ if }n \equiv 0\pmod{p^{\lambda-1}}, \\
         0 & \text{ otherwise.}
    \end{dcases}
\end{equation}
By a similar method (instead replacing $d$ by $d+4$) we have
\begin{equation}\label{eq:twisted-exp-sum-2}
    \sum_{d(2^\lambda)^\times} \pfrac{-1}{d} e_{2^\lambda}(d n) = 
    \begin{dcases}
        2^{\lambda-1}i \pfrac{-4}{n/2^{\lambda-2}} & \text{ if }n \equiv 0\pmod{2^{\lambda-2}}, \\
         0 & \text{ otherwise.}
    \end{dcases}
\end{equation}

\section{Proof of Theorems~\ref{thm:exp-sum-identity-g=1} and \ref{thm:exp-sum-general} } \label{sec:proof-identity}

Fix $\alpha,\beta,m,n$ satisfying $\frac{m}{2\Delta}-q(\alpha)\in \Z$ and $\frac{n}{2\Delta}-q(\beta)\in \Z$.
Suppose that $(-1)^{(g-1)/2}m$ is a fundamental discriminant.
Let $h = \frac{g+1}{2}$ and for convenience set
\[
    \chi_m \left(Nc, \ell, \mfrac{\ell^2-mn}{4Nc}\right) = 0 \quad \text{ if }\ell^2\not\equiv mn\pmod{4Nc}.
\]
Let
\begin{equation}
    L_v(c) = i^\sigma \sqrt{2N}\sum_{u\mid (v,c)} \pfrac{m_L}{u} \sqrt{\frac uc} \, S_{\alpha \frac vu, \beta}\left( mv^2/u^2, n, \mfrac cu \right)
\end{equation}
and
\begin{equation}
    R_v(c) =
    \begin{dcases}
    \sum 
    \chi_{m}(Nc,\ell,\tfrac{\ell^2-mn}{4Nc}) e_{\Delta c}(\ell v) & \text{ if }g=1, \\
    \sum 
    \xi_{\alpha,\beta}(\ell,m,n,c) e_{\Delta c}(\ell v) & \text{ if }g>1,
    \end{dcases}
\end{equation}
where $\ell$ runs mod $2Nc$ with $\frac{\ell}{\Delta}-\vprod{\alpha,\beta}\in \Z$ in both sums.
Note that $R_v(c)$ is periodic in $v$ with period $2Nc$, and its Fourier transform equals
    \begin{equation} \label{R-fourier}
    \frac{1}{2Nc}\sum_{v(2Nc)} e_{\Delta c}(-v\ell) R_v(c) = 
    \begin{cases}
        \chi_m (Nc,\ell,\tfrac{\ell^2-mn}{4Nc})  & \text{ if } \tfrac{\ell}{\Delta} \equiv \vprod{\alpha,\beta} (1) \text{ and } g=1, \\
        \xi_{\alpha,\beta}(\ell,m,n,c) & \text{ if }\tfrac{\ell}{\Delta} \equiv \vprod{\alpha,\beta} (1) \text{ and } g>1,\\
        0 & \text{ otherwise}.
    \end{cases}
\end{equation}
We claim that $L_v(c)$ is also periodic in $v$ with period $2Nc$.
After inserting the definition of the Kloosterman sum \eqref{eq:kloo-weil-def} and the formula for the coefficients of the Weil representation \eqref{eq:rho-alpha-beta-formula} into the definition of $L_v(c)$, we obtain
\begin{multline} \label{eq:Lv-three-sums}
    L_v(c) = \sum_{u\mid (v,c)}  \pfrac{m_L}{u} (c/u)^{-h} \sum_{d(c/u)^\times} 
    \\ 
    \times \sum_{r\in L/ (c/u) L} e_{c/u}\left( \left(\mfrac{m(v/u)^2}{2\Delta}-q(\alpha v/u+r)\right)a +\vprod{\beta,\alpha v/u+r} + \left(\mfrac n{2\Delta} - q(\beta)\right)d \right),
\end{multline}
where $ad\equiv 1\pmod{c/u}$.
Since $4Nq(\alpha)\in \Z$ and $2N\alpha\in L$ we have
\begin{align}
    \vprod{\alpha v/u+r,2N(c/u)\alpha} &= (c/u)\left(4Nq(\alpha)(v/u)+2N\vprod{r,\alpha}\right) \equiv 0\pmod{c/u}, \\
    q(2N\alpha c/u) &= 4N^2q(\alpha)(c/u)^2 \equiv 0\pmod{c/u}, \\
    \vprod{\beta,2N\alpha(c/u)} &= \vprod{\beta,2N\alpha}(c/u) \equiv 0\pmod{c/u}.
\end{align}
Thus $L_v(c)$ is indeed periodic in $v$ with period $2Nc$.
So it suffices to prove that
\begin{equation} \label{eq:cal-l}
    \mathcal L_\ell(c) := \frac{1}{2Nc}\sum_{v(2Nc)} e_{\Delta c}(-v\ell) L_v(c)
\end{equation}
agrees with the right-hand side of \eqref{R-fourier} for all $\ell \in \Z$.

By \eqref{eq:cal-l} and \eqref{eq:Lv-three-sums}, the quantity $\mathcal L_\ell(c)$ comprises four sums
\begin{equation}
    \sum_{v(2Nc)} \sum_{u\mid (v,c)}  \sum_{d(c/u)^\times} \sum_{r\in L/ (c/u) L}
\end{equation}
which we reorder as
\[
    \sum_{u\mid c} \sum_{d(c/u)^\times} \sum_{\substack{v(2Nc) \\ u\mid v}} \sum_{r\in L/(c/u)L}.
\]
We replace $v$ by $uv$, then $u$ by $c/u$, and rearrange terms to obtain
\begin{multline}
    \mathcal L_\ell(c) = \frac{1}{2Nc} \sum_{u\mid c} \pfrac{m_L}{c/u} u^{-h}\sum_{d(u)^\times} e_u\left(\tilde n d\right)
    \\ 
    \times \sum_{v(2Nu)} \sum_{r\in L/ uL} e_{u}\left( a\tilde mv^2 - a\vprod{\alpha v,r} - aq(r) +\vprod{\beta,r} + (\vprod{\beta,\alpha} - \tfrac {\ell}{\Delta})v\right), 
\end{multline}
where $\tilde m = \tfrac{m}{2\Delta}-q(\alpha) \in \Z$ and $\tilde n = \tfrac{n}{2\Delta}-q(\beta)\in \Z$.
If we make the change of variable $v\mapsto v+u$ we see that the $v$-sum equals zero unless
\[
    \tfrac{\ell}{\Delta} \equiv \vprod{\alpha,\beta} \pmod{1}.
\]
For the remainder of this proof we make this assumption and we set $\tilde \ell = \tfrac{\ell}{\Delta} - \vprod{\alpha,\beta}$.
By Lemma~\ref{lem:l^2-mn} we have $\ell^2\equiv mn\pmod{2N}$.

Now the summands in the $v$-sum are invariant under $v\mapsto v+u$, so we can write
\begin{equation} \label{eq:L-ell-c-u-d-sum}
    \mathcal L_\ell(c) = c^{-1} \sum_{u\mid c}\pfrac{m_L}{c/u} u^{-h}\sum_{d(u)^\times} e_u\left(\tilde n d\right) \mathcal S(d,u),
\end{equation}
where
\begin{equation}
    \mathcal S(d,u) = \sum_{v(u)} \sum_{r\in L/ uL} e_{u}\left( a\tilde mv^2 - a\vprod{\alpha v,r} - aq(r) +\vprod{\beta,r} -\tilde \ell v\right).
\end{equation}
Since $(d,u)=1$ we can replace $v$ by $dv$ and $r$ by $dr$ to get
\begin{equation}
    \mathcal S(d,u) = \sum_{v(u)} \sum_{r\in L/ uL} e_{u}(df(v,r)),
\end{equation}
where
\begin{equation}
    f(v,r) = \tilde m v^2 - \vprod{\alpha,r}v - q(r) + \vprod{\beta,r}-\tilde \ell v.
\end{equation}

\begin{remark*}
    In the case $g=1$, the two-dimensional quadratic Gauss sum $\mathcal S(d,u)$ is analogous to the sum appearing in Proposition~2 of \cite{GKZ}.
\end{remark*}

\begin{lemma} \label{lem:p-div-ell}
    We have $\mathcal S(d,u)=0$ unless $(m,u)\mid \ell$.
\end{lemma}
\begin{proof}
    Let $w=\Delta u/(m,u)$ and note that $\alpha w\in L$.
    Then, since $\Delta$ is even,
    \begin{align}
        f(v+w,r-\alpha w) &= f(v,r) + \tfrac 12 \Delta mu^2/(m,u)^2 + (mv-\ell)u/(m,u) \\
        &\equiv f(v,r) - \ell u/(m,u) \pmod{u}.
    \end{align}
    Thus $\mathcal S(d,u) = e(-d\ell/(m,u)) \mathcal S(d,u)$, i.e.~$\mathcal S(d,u) = 0$ unless $(m,u)\mid \ell$.
\end{proof}

Using the Ramanujan sum evaluation
\begin{equation}
    \sum_{d(u)^\times} e_u(yd) = \sum_{t\mid (u,y)} \mu(u/t) t
\end{equation}
we find that
\begin{equation}
    \sum_{d(u)^\times} e_u(\tilde n d) \mathcal S(d,u) = \sum_{v,r} \sum_{t\mid (u,y)}\mu(u/t) t = \sum_{t\mid u} \mu(u/t)t \sum_{\substack{v,r \\ t\mid y}} 1,
\end{equation}
where $y=f(v,r)+\tilde n$.
The inner sum is invariant under $v\mapsto v+t$ and $r\mapsto r+t$, so we get
\begin{equation}
    \sum_{d(u)^\times} e_u(\tilde n d) \mathcal S(d,u) = u^{g+1} \sum_{t\mid u} \mu(u/t) t^{-g} N(t),
\end{equation}
where
\begin{equation} \label{eq:Nt-def}
    N(t) = \# \left\{(v,r)\in \Z/t\Z \times L/tL : f(v,r)+\tilde n \equiv 0 \pmod{t}\right\}.
\end{equation}
By the Chinese remainder theorem, $N(t)$ is multiplicative, therefore $L_\ell(c)$ is multiplicative as a function of $c$.
For the remainder of this section, assume $c=p^\lambda$ where $p$ is prime and $p\mid \ell$ if $p\mid m$ (which we can assume by Lemma~\ref{lem:p-div-ell}). 
By the discussion above, we have two valid expressions for $\mathcal L_\ell(p^\lambda)$:
\begin{align}
    \mathcal L_\ell(p^\lambda) 
    &= p^{-\lambda} \sum_{j=0}^\lambda \pfrac{m_L}{p^{\lambda-j}} p^{-jh} \sum_{d(p^j)^\times} e_{p^j}(\tilde n d) \mathcal S(d,p^j) \label{eq:exp-1} \\
    &= p^{-\lambda} \sum_{j=0}^\lambda \pfrac{m_L}{p^{\lambda-j}} p^{-j(h-1)} \left(N(p^j) - p^g N(p^{j-1})\right). \label{eq:exp-2}
\end{align}

We proceed by cases as follows.
\begin{enumerate}
    \item We first assume that $(m,\Delta,p)=1$ and that $p$ is odd, and we evaluate the Gauss sums $\mathcal S(d,u)$ in each of the cases $p\nmid m$ and $p\mid m$.
    \item In the case $g=1$, we evaluate $\mathcal S(d,u)$ for the remaining ``bad'' primes.
    \item When $g>1$ we approach the problem by studying the counting function $N(p^j)$. 
    We first show that the quantity $N(p^j)-p^g N(p^{j-1})$ is frequently zero. We then estimate the size of $N(p^j)$.
\end{enumerate}

\subsection{The case $(m,\Delta,p)=1$ with $p$ odd} \label{sec:complete-square}
We would like to make a change of variables that eliminates the linear terms in $f(v,r)$ modulo $u$, where $u=p^j$.
For $w\in \Z$ and $s\in L$ we have
\begin{multline}
    f(v+w,r+s) 
    = \tilde m v^2 
    - \langle \alpha,r \rangle v 
    - q(r)  
    + \frac{mw^2-2\ell w}{2\Delta}+q(\beta) - q(\beta-\alpha w-s)
    \\ + v\left( \frac{mw-\ell}{\Delta} - \vprod{\alpha,\alpha w+s-\beta}\right) 
    - \langle \alpha w+ s-\beta,r \rangle.
\end{multline}
If we are to eliminate the terms on the second line, then a natural choice is $s=\beta-\alpha w$, but this is usually not an element of $L$.
Note that either $p\nmid m$ or $p\parallel m$ because $(-1)^{(g-1)/2}m$ is a fundamental discriminant.

\begin{lemma} \label{lem:w-def}
    Suppose that $p$ is odd and $(m,\Delta,p)=1$.
    Let $k\in \Z^+$ and let
    \begin{equation}
        w = 
        \begin{cases} \label{eq:w-def}
            \bar m \ell & \text{ if }p\nmid m, \\
            \overline{(m/p)}(\ell/p) & \text{ if }p\mid m,
        \end{cases}
    \end{equation}
    where $\bar m m\equiv 1\pmod{p^{k}}$ in the first case and $\overline{(m/p)}(m/p)\equiv 1\pmod{p^{k}}$ in the second case.
    Then $\beta-\alpha w\in L+p^{k} (L'/L)$.
\end{lemma}

\begin{proof}
    We begin with the observation
    \[
    m\beta - \ell\alpha = \tilde m(2\Delta\beta) - \tilde \ell(\Delta\alpha) + \Delta(\vprod {\alpha,\alpha}\beta - \vprod{\alpha,\beta}\alpha) \in L
    \]
    by Lemma~\ref{lem:alpha-beta-gamma}.
    By Lemma~\ref{lem:p-div-ell} we have $(m,p)\mid \ell$, so we can write $m=(m,p) m_1$ and $\ell=(m,p) \ell_1$.
    We claim that if $p\mid m$ then $m_1\beta-\ell_1\alpha\in L$.
    Indeed, the lattice element
    \[
    \Delta(m\beta-\ell\alpha) = p(m_1(\Delta\beta)-\ell_1(\Delta\alpha))
    \]
    is an element of $\Delta L\cap pL=\Delta pL$ because, in this case, $p\nmid \Delta$.
    Thus $m_1\beta-\ell_1\alpha\in L$.
    Let $w=\overline{m}_1 \ell_1$, as in \eqref{eq:w-def}. Then 
    \[
    \beta-\alpha w - \overline{m}_1(m_1\beta-\ell_1\alpha) = (1-m_1\overline{m}_1)\beta\in p^{k} L'/L.
    \]
    The statement of the lemma follows.
\end{proof}

Write $2\Delta=p^\nu \Delta'$, with $p\nmid \Delta'$, and make the change of variable $d\mapsto \Delta'd$ in \eqref{eq:exp-1}.
Choose $w$ as in \eqref{eq:w-def} with $k=\lambda+\nu$ and use Lemma~\ref{lem:w-def} to choose $s\in L$ and $\gamma\in L'/L$ such that $s-\beta+\alpha w = p^{\lambda+\nu}\gamma$.
Then 
\[
    \Delta'\vprod{\alpha,\alpha w+s-\beta} = p^\lambda \vprod{2\Delta\alpha,\gamma} \equiv 0\pmod{p^\lambda},
\]
and a similar statement holds for $\Delta'\vprod{\alpha w+s-\beta,r}$ and $\Delta' q(\beta-\alpha w-s)$.
Furthermore, we have $\Delta'(mw-\ell)/\Delta \equiv 0\pmod{p^\lambda}$.
Therefore
\[
    \Delta'(f(v+w,r+s)+\tilde n) \equiv \Delta'(\tilde m v^2 - \vprod{\alpha,r}v-q(r)) + \hat n \pmod{p^\lambda},
\]
where (recalling that $\nu=0$ if $p\mid m$)
\begin{align}
    \hat n = \Delta' \frac{mw^2-2\ell w+n}{2\Delta}  \equiv
    \begin{dcases}
        \frac{n-\bar m \ell^2}{p^\nu} \pmod{p^\lambda} & \text{ if } p \nmid m, \\
        n-\bar{(m/p)}\ell^2/p \pmod{p^\lambda} & \text{ if }p\mid m.
    \end{dcases}
\end{align}
In the case $p\nmid m$ this shows that $\ell^2\equiv mn\pmod{p^\nu}$ because $\hat n$ must be an integer.

Thus
\begin{equation} \label{eq:n1-S1}
    \sum_{d(u)^\times} e_u(\Delta'd\tilde n) \mathcal S(\Delta'd,u) = \sum_{d(u)^\times} e_u(d\hat n) \mathcal S_1(d,u),
\end{equation}
where
\begin{equation}
    \mathcal S_1(d,u) = \sum_{v(u)}\sum_{r\in L/uL} e_u(d\Delta'f_1(v,r)), \qquad f_1(v,r) = \tilde m v^2 - \vprod{\alpha,r}v-q(r).
\end{equation}
Let $M$ denote the Gram matrix of $L$ and identify $L$ with $\Z^g$ so that $\vprod{x,y} = x^TMy$ and $\alpha = M^{-1}a$ for some $a\in \Z^g$. 
Then we can write $f_1(v,r) = \frac 12 x^T S x$, where $x=(v,r)\in \Z^{g+1}$ and $S$ is the block matrix
\[
    S = \ptMatrix{2\tilde m}{-a^T}{-a}{-M}.
\]
The determinant of $S$ equals $\frac{m}{\Delta}\det(-M) = -m$.

Suppose first that $p\nmid m$. 
Then by \eqref{eq:Gauss-sum-quadratic-form}, applied to the quadratic form $x\mapsto d\Delta' x^TSx$ on the lattice $\Z\oplus L$, together with \eqref{eq:gauss-odd}, we have
\begin{equation}
    \mathcal S_1(d,u) = \pfrac{-m}{u} G(u)^{g+1} = \pfrac{m_L}{u} u^{h},
\end{equation}
where we have used the fact that $g$ is odd so $(\bar 2d\Delta')^{g+1}$ is a square.
It follows that
\begin{align}
    \mathcal L_\ell(c) &= \frac{1}{c} \pfrac{m_L}{c} \sum_{u\mid c} \sum_{d(u)^\times} e_u(\hat nd) 
    = \frac{1}{c} \pfrac{m_L}{c} \sum_{u\mid c} \sum_{k\mid (u,\hat n)} \mu(u/k)k \\
    &= \pfrac{m_L}{p^\lambda} \times 
    \begin{cases}
        1 & \text{ if }\hat n \equiv 0\pmod{p^\lambda}, \\
        0 & \text{ otherwise}.
    \end{cases}
\end{align}
The condition $\hat n \equiv 0 \pmod{p^\lambda}$ is equivalent to $\ell^2\equiv mn\pmod{p^{\lambda+\nu}}$.

Now suppose that $p\mid m$.
Note that in this case many of the terms in \eqref{eq:L-ell-c-u-d-sum} are zero, so
\[
    \mathcal L_\ell(c) = c^{-h-1}  \sum_{d(c)^\times} e_c(d\hat n) \mathcal S_1(d,c).
\]
By assumption we have that $p\nmid \Delta$, so $2\Delta=\Delta'$ and
we can write
\[
    \Delta'f_1(v,\bar{\Delta}r) \equiv mv^2 - 2\bar{\Delta}q(\hat \alpha v+r) \pmod{c},
\]
where $\Delta \bar \Delta \equiv 1\pmod{c}$ and $\hat\alpha = \Delta\alpha \in L$.
Thus
\begin{equation}
    \mathcal S_1(d,c) = \sum_{v(c)}e_c(dmv^2)\sum_{r\in L/cL} e_c(-2d \bar \Delta q(r)) = p\sum_{v(c/p)}e_{c/p}(d(m/p)v^2)\sum_{r\in L/cL} e_c(-2d \bar \Delta q(r)).
\end{equation}
Note that $(m/p,c/p)=1$, so by \eqref{eq:Gauss-y} and \eqref{eq:gauss-odd}, the first Gauss sum evaluates to
\[
    \sum_{v(c/p)}e_{c/p}(d(m/p)v^2) = \pfrac{d(m/p)}{c/p} \ep_{c/p} (c/p)^{\frac 12}.
\]
For the second, since $p\nmid \Delta$, 
\eqref{eq:Gauss-sum-quadratic-form} yields
\[
    \sum_{r\in L/cL} e_c(-2d \bar \Delta q(r)) = \pfrac{(-d\bar\Delta)^g\Delta}{c} \ep_c^g c^{\frac g2}.
\]
Therefore, since $g$ is odd,
\begin{align}
    \mathcal S_1(d,c) &= p^{\frac 12} \pfrac dp \pfrac{-1}{c} \pfrac{m/p}{c/p} \ep_{c/p} \ep_c^g c^h,
\end{align}
from which it follows that
\[
    \mathcal L_\ell(c) = p^{\frac 12}c^{-1} \pfrac{m/p}{c/p} \pfrac{-1}{c} \ep_{c/p}\ep_c^g \sum_{d(c)^\times} \pfrac dp e_c(d\hat n).
\]
By \eqref{eq:twisted-exp-sum} we find that
$\mathcal L_\ell(c)=0$ unless $\hat n\equiv 0\pmod{p^{\lambda-1}}$, which we now assume. 
Then
\[
    \mathcal L_\ell(c) =  \pfrac{(-1)^{(g+1)/2}}{p^\lambda} \pfrac{m/p}{c/p}  \pfrac{-\hat n/p^{\lambda-1}}{p},
\]
because $\ep_p\ep_{c/p}\ep_c^g = (\frac{-1}{p})^{\lambda(g-1)/2+1}$.
We have 
\[
    -\frac{\hat n}{p^{\lambda-1}} = -\frac{n-\overline{(m/p)}\ell^2/p}{p^{\lambda-1}} \equiv  \overline{(m/p)} \frac{\ell^2-mn}{p^\lambda} \pmod{p},
\]
so
\[
    \mathcal L_\ell(c) = \pfrac{-m_L/p}{p^\lambda} \pfrac{(\ell^2-mn)/p^{\lambda}}{p} = \pfrac{m_L/p^*}{p^\lambda} \pfrac{p^*}{(\ell^2-mn)/p^\lambda}.
\]
Lastly, we note that the condition $\hat n \equiv 0\pmod{p^{\lambda-1}}$ is equivalent to $\ell^2\equiv mn\pmod{p^\lambda}$.

\subsection{The case $g=1$}\label{sec:rank-1}

In this section we evaluate $\mathcal L_\ell(c)$, with $c=p^\lambda$, in the remaining cases when $g=1$: $p=2$ or $(m,\Delta,p)>1$.
To match the setup of Theorem~\ref{thm:exp-sum-identity-g=1}, we take $L=\Z$ with $\vprod{x,y} = \Delta xy$ and $\alpha = \frac{a}{\Delta}$ and $\beta=\frac b{\Delta}$.

Suppose first that $p$ is odd and $p\mid (m,\Delta)$.
Then by \eqref{eq:exp-1} we have
\[
    \mathcal L_\ell(c) = c^{-2} \sum_{d(c)^\times} e_c(d\tilde n) \mathcal S(d,c).
\]
Since $m\equiv a^2 \pmod{4N}$,
we have $p\mid a$, and
since $p\parallel m$, we cannot have $p^2\mid \Delta$ (i.e.~$\nu=1$).
By replacing $r$ with $r+c/p$ in $\mathcal S(d,c)$, we find that $\mathcal S(d,c)=0$ unless $p\mid b$, and thus $p\mid n$ because $\frac n{2\Delta}-q(\beta) \in \Z$.
In what follows we make these assumptions and write
\[
    m = pm_1, \quad \Delta = p\Delta_1, \qquad (p,m_1)=(p,\Delta_1)=1,
\]
and similarly define $\ell_1$, $a_1$, $b_1$, and $n_1$.
We will need the following analogue of Lemma~\ref{lem:w-def}.

\begin{lemma} \label{lem:w-def-g=1}
    Suppose that $g=1$ and that $(m,\Delta,p)>1$.
    For $k\in \Z^+$, if
    \begin{equation}
        w \equiv  
            \bar m_1 \ell_1 \pmod{p^k}
    \end{equation}
    then $\beta-\alpha w\in L+p^{k} (L'/L)$.
\end{lemma}

\begin{proof}
    Since $g=1$ we have $\vprod{\alpha,\alpha}\beta-\vprod{\alpha,\beta}\alpha = 0$, so
    $m\beta-\ell\alpha = 2\tilde m b - \tilde \ell a \in pL$.
    It follows that $m_1\beta-\ell_1\alpha\in L$.
    The remainder of the proof follows the proof of Lemma~\ref{lem:w-def}.
\end{proof}

The discussion following Lemma~\ref{lem:w-def} shows that 
\begin{equation}
    \mathcal L_\ell(c) = c^{-2} \sum_{d(c)^\times} e_c(d\hat n) \mathcal S_1(d,c),
\end{equation}
with $\mathcal S_1(d,c)$ as in Section~\ref{sec:complete-square} and
\begin{equation}
    \hat n \equiv n_1 - \bar m_1 \ell_1^2 \pmod{p^\lambda}.
\end{equation}
Note that $\Delta'=2\Delta_1$ and
\begin{equation}
    f_1(v,r) = \tilde m v^2 - arv - \tfrac 12 \Delta r^2.
\end{equation}
We have $\Delta' f_1(v,r) = m_1 v^2 - p(a_1 v + \Delta_1 r)^2$, so after a change of variables we obtain
\begin{equation}
    \mathcal L_\ell(c) = pc^{-2}\sum_{d(c)^\times} e_c(d\hat n)\sum_{v(c)} e_c(dm_1 v^2) \sum_{r(c/p)} e_{c/p}(-dr^2).
\end{equation}
The rest of the computation resembles the case $p\mid m$ in Section~\ref{sec:complete-square}.
Using \eqref{eq:gauss-odd}, \eqref{eq:Gauss-y}, and \eqref{eq:twisted-exp-sum}, we find that $\mathcal L_\ell(c) = 0$ unless $p^{\lambda-1}\mid \hat n$, in which case we have
\begin{align}
    \mathcal L_\ell(c) &= \pfrac{m_1}{p^\lambda} \pfrac{-1}{p^{\lambda-1}} \pfrac{\hat n/p^{\lambda-1}}{p} \ep_{p^\lambda} \ep_{p^{\lambda-1}} \ep_p \\
    &= \pfrac{m_L/p^*}{p^{\lambda+1}}\pfrac{p^*}{(\ell^2-mn)/p^{\lambda+1}}.
\end{align}
In this case the condition $p^{\lambda-1}\mid \hat n$ is equivalent to $\ell^2\equiv mn\pmod{p^{\lambda+1}}$.

Now suppose that $p=2$ and $m$ is odd.
Since $(m,\Delta,p)=1$, we follow Section~\ref{sec:complete-square} to get
\begin{equation}
    \mathcal L_\ell(2^\lambda) = 2^{-\lambda} \pfrac{m}{2^\lambda} \sum_{j=0}^\lambda \pfrac{m}{2}^j  2^{-j} \sum_{d(2^j)^\times} e_{2^j}(d\hat n) \mathcal S_1(d,2^j),
\end{equation}
where $\hat n  = (n-\bar m \ell^2)/2^{\nu}$ and
\begin{equation}
    \mathcal S_1(d,2^j) = \sum_{v(2^j)}\sum_{r(2^j)} e_{2^j}(d\Delta'f_1(v,r)), \qquad f_1(v,r) = \tilde m v^2 - a r v - \tfrac 12 \Delta r^2.
\end{equation}
The congruence $m\equiv a^2 \pmod{4N}$ shows that $a$ is odd and that $m \equiv 1\pmod{4}$.
By Corollary~3.1 of \cite{alaca-doyle} we have\footnote{There are several cases to tediously check, but all yield the same result. Alternatively, one can prove this using several applications of \eqref{eq:gauss-even}.}
\begin{equation}
    \pfrac{m}{2}^j 2^{-j} \mathcal S_1(d,2^j) = 1
\end{equation}
when $(d,2^j)=1$,
therefore
\begin{align}
    \mathcal L_\ell(2^\lambda) &= 2^{-\lambda} \pfrac{m}{2^\lambda} \sum_{j=0}^\lambda \sum_{d(2^j)^\times} e_{2^j}(d\hat n) = 2^{-\lambda} \pfrac{m}{2^\lambda} \sum_{d(2^\lambda)} e_{2^\lambda}(d\hat n) \\
    &=
    \pfrac{m_L}{2^\lambda}
    \begin{dcases}
        1 & \text{ if } \hat n \equiv 0 \pmod{2^\lambda}, \\
        0 & \text{ otherwise}.
    \end{dcases}
\end{align}
Here the condition $\hat n \equiv 0 \pmod{2^\lambda}$ is equivalent to $\ell^2\equiv mn \pmod{2^{\lambda+\nu}}$.

Finally, suppose that $p=2$ and that $m$ is even.
Then
\[
    \mathcal L_\ell(2^\lambda) = 2^{-2\lambda}  \sum_{d(2^\lambda)^\times} e_{2^\lambda}(d\tilde n) \mathcal S(d,2^\lambda).
\]
Define $\mu$ by $2^{\mu}\parallel m$ and recall that $\nu$ satisfies $2^\nu \parallel 2\Delta$.
Since $m$ is a fundamental discriminant, we have $\mu\in \{2,3\}$. 
Furthermore, $a$ is even and $\frac m4 \equiv \frac{a^2}{4} \pmod {\frac 12 \Delta}$.
Since $\frac m4 \equiv 2,3\pmod{4}$, we see that $4\nmid \frac 12 \Delta$.
In other words, $\nu \in \{2,3\}$.
In what follows, we assume that $\lambda\geq 3$; when $\lambda\in \{1,2\}$ there are only finitely many cases to check.

Suppose first that $\mu=\nu=2$.
Then $m_1:=\frac m4$ and $D:=\frac 12\Delta$ are odd.
It follows from two applications of \eqref{eq:gauss-even} that $\mathcal S(d,2^\lambda)=0$ unless $b$ is even, in which case we have
\begin{align}
    \mathcal S(d,2^\lambda) 
    &= (1+i)\ep_{-dD}^{-1}\pfrac{2^\lambda}{-dD} 2^{\frac \lambda2} e_{2^\lambda}(d\bar D b_1^2) \sum_{v(2^\lambda)} e_{2^\lambda}(d\bar D(m_1v^2-\ell_1 v)) \\
    &= (1+i)^2 \ep_{-dD}^{-1} \ep_{dDm_1}^{-1} \pfrac{2^\lambda}{-m_1} e_{2^\lambda}(d\bar D (b_1^2-\bar m_1 \ell_2^2))2^\lambda,
\end{align}
where $b=2b_1$ and $\ell=2\ell_1=4\ell_2$.
(Note that $4\mid\ell$ by Lemma~\ref{lem:p-div-ell}.)
Since $m_1\equiv 3\pmod{4}$ we have $\ep_{-dD}^{-1} \ep_{dDm_1}^{-1} = -\pfrac{-1}{dD}$.
Also, $4\mid n$ because $b$ is even, so we can write $n=4n_1$.
By replacing $d$ by $Dd$ and applying \eqref{eq:twisted-exp-sum-2}, we obtain
\begin{align}
    \mathcal L_\ell(2^\lambda) 
    &= -2^{1-\lambda}i \pfrac{2^\lambda}{-m_1} \sum_{d(2^\lambda)^\times} \pfrac{-1}{d} e_{2^\lambda}(d(n_1-\bar m_1\ell_2^2)) \\
    &= 
    \begin{dcases}
        \pfrac{2^\lambda}{-m_1} \pfrac{-4}{(n_1-\bar m_1\ell_2^2)/2^{\lambda-2}} & \text{ if } n_1 \equiv \bar m_1 \ell_2^2 \pmod{2^{\lambda-2}}, \\
        0 & \text{ otherwise}.
    \end{dcases}
\end{align}
We conclude that $\mathcal L_\ell(2^\lambda)=0$ unless $\ell^2 \equiv mn \pmod{2^{\lambda+2}}$, in which case
\begin{equation}
    \mathcal L_\ell(2^\lambda) = \pfrac{m/2^*}{2^\lambda} \pfrac{2^*}{(\ell^2-mn)/2^{\lambda+2}}.
\end{equation}

If $(\mu,\nu)=(3,2)$ then by a similar argument we obtain
\begin{equation}
    \mathcal S_1(d,2^\lambda) = \ep_{m_2}^{-1} \pfrac{2^\lambda}{-m_2} \pfrac{2^*}{dD} e_{2^\lambda}(d\bar D(b_1^2-2\bar m_2\ell_3^2)) 2^{\lambda+\frac 32},
\end{equation}
where $m=8m_2$ and $\ell=8\ell_3$.
Thus
\begin{equation}
    \mathcal L_\ell(2^\lambda) = 
    \begin{dcases}
        \pfrac{2^{\lambda+1}}{-m_2} \pfrac{2^*}{(n_1^2-2\bar m_2\ell_3^2)/2^{\lambda-3}} & \text{ if } n_1^2\equiv 2\bar m_2 \ell_3^2 \pmod{2^{\lambda-3}}, \\
        0 & \text{ otherwise}.
    \end{dcases}
\end{equation}
In other words, $\mathcal L_\ell(2^\lambda)=0$ unless $\ell^2 \equiv mn \pmod{2^{\lambda+2}}$, in which case we have
\[
    \mathcal L_\ell(2^\lambda) = \pfrac{m/2^*}{2^\lambda} \pfrac{2^*}{(\ell^2-mn)/2^{\lambda+2}}.
\]

The remaining cases $(\mu,\nu)=(2,3)$ and $(\mu,\nu)=(3,3)$ are similar to the previous two, except that $\frac 12 \Delta$ is even and $\tilde m$ is odd, so we evaluate the $v$-sum first, then the $r$-sum.

\subsection{The case when $g>1$ and either $p=2$ or $p$ is odd and $(m,\Delta,p)>1$}

\label{sec:counting}

In each of these cases we have
\begin{equation}
    \mathcal L_\ell(p^\lambda) = p^{-\lambda h} \left(N(p^\lambda)-p^g N(p^{\lambda-1})\right),
\end{equation}
and we will show that
\begin{enumerate}
    \item $\mathcal L_\ell(p^\lambda)=0$ unless $\ell^2\equiv mn\pmod{2Np^{2\lf}}$, and
    \item $|\xi_{\alpha,\beta}(\ell,m,n,p^\lambda)| \leq p^{A\nu g}$ for some absolute constant $A$.
\end{enumerate}
Then, because $|\xi_{\alpha,\beta}(\ell,m,n,p^\lambda)|\leq 1$ for all other primes, we conclude that
\begin{equation}
    |\xi_{\alpha,\beta}(\ell,m,n,c)| \leq \prod_{\substack{p\mid 2(\Delta,c) \\ p^\nu \parallel 2\Delta}} p^{A\nu g} \leq |2\Delta|^{Ag} \ll_L 1.
\end{equation}

We begin with a simple observation to motivate our approach.
If $k\geq 1$ and $(v_0,r_0)$ is a solution to the congruence $f(v,r)+\tilde n \equiv 0\pmod{p^{k-1}}$, then $(v_0+p^{k-1}x,r_0+p^{k-1}y)$ is a solution to $f(v,r)+\tilde n \equiv 0 \pmod{p^k}$ if and only if
\begin{equation} \label{eq:x-y-mod-p}
    (2\tilde m v_0 - \vprod{\alpha,r_0}-\tilde \ell)x + \vprod{\alpha v_0+r_0-\beta,y} \equiv -\frac{f(v_0,r_0)+\tilde n}{p^{k-1}} \pmod{p}.
\end{equation}
If $p\nmid (2\tilde mv_0-\vprod{\alpha,r_0}-\tilde \ell)$ then there are $p^g$ pairs $(x,y)$ satisfying \eqref{eq:x-y-mod-p}: for each $y \in L/pL$ there is exactly one $x\in \Z/p\Z$ for which \eqref{eq:x-y-mod-p} holds.
A similar argument shows that there are $p^g$ pairs satisfying \eqref{eq:x-y-mod-p} as long as there exists a $y\in L/pL$ such that $\vprod{\alpha v_0+r_0-\beta,y}$ is not divisible by $p$.
Thus $N(p^k)-p^g N(p^{k-1})=0$ unless there is a solution $(v_0,r_0)$ for which $2\tilde mv_0-\vprod{\alpha,r_0}-\tilde \ell \equiv 0\pmod{p}$ and $\vprod{\alpha v_0+r_0-\beta,y} \equiv 0\pmod{p}$ for all $y\in L/pL$.

Generalizing this idea, for fixed $\lambda$ and for $j\leq k\leq \lambda$ let $\mathcal M_j(p^k)$ be the set of pairs $(v,r)$ with $v\in \Z/p^\lambda\Z$ and $r\in L/p^\lambda L$ such that $(v,r)$ is a solution to the congruences
\begin{align*}
    f(v,r)+\tilde n &\equiv 0 \pmod{p^k}, \\
    2\Tilde{m}v-\vprod{\alpha,r}-\Tilde{\ell} &\equiv 0\pmod{p^j}, \\
    \vprod{\alpha v+r-\beta,y} &\equiv 0 \pmod{p^j} \qquad \forall y\in L.
\end{align*}
Then $\mathcal M_{j+1}(p^k) \subseteq \mathcal M_j(p^k)$ for each $j\leq k-1$.
Let $\mathcal M^*_k(p^k) = \mathcal M_k(p^k)$ and, 
for $j\leq k-1$, let $\mathcal M^*_j(p^k)= \mathcal M_j(p^k)\setminus \mathcal M_{j+1}(p^k)$. 
We write 
\[
M_j(p^k) = \#\mathcal M_j(p^k) \quad \text{ and } \quad M_j^*(p^k) = \#\mathcal M_j^*(p^k).
\]
Then we have
\begin{equation}
    N(p^k) = p^{-(g+1)(\lambda-k)}M_0(p^k),
\end{equation}
so
\begin{equation} \label{eq:N-M-translation}
    N(p^\lambda) - p^g N(p^{\lambda-1}) = M_0(p^\lambda) - \tfrac 1p M_0(p^{\lambda-1}).
\end{equation}

\begin{lemma} \label{lem:N=M-lambda/2}
    Notation as above,
    we have 
    \begin{equation} \label{eq:N=M-lambda/2}
    N(p^\lambda)-p^gN(p^{\lambda-1}) = M_{\lf}(p^\lambda)-\tfrac{1}{p}M_{\lf}(p^{\lambda-1}).
    \end{equation}
\end{lemma}
\begin{proof}
By \eqref{eq:N-M-translation} we have
\begin{equation}
    N(p^\lambda)-p^gN(p^{\lambda-1}) = \sum_{j=0}^{\lambda}\left(M^*_j(p^\lambda)-\tfrac{1}{p}M^*_j(p^{\lambda-1})\right),
\end{equation}
so to prove \eqref{eq:N=M-lambda/2} it suffices to show that
\begin{equation} \label{eq:Mj-want}
    M_j^*(p^\lambda) = \tfrac 1p M_j^*(p^{\lambda-1}) \quad \text{ for all }j<\ltf.
\end{equation}

Suppose that $j<\lf$; then $2j+2\leq \lambda$.
Let $(v,r)\in \mathcal M^*_j(p^{\lambda-1})$.
We claim that for every $x\in \Z/p^{j+1}\Z$ and every $y\in L/p^{j+1}L$,
\begin{equation} \label{eq:new-xy-in-M-j}
    (v+xp^{\lambda-j-1},r+yp^{\lambda-j-1}) \in \mathcal M_j^*(p^{\lambda-1}).
\end{equation}
Indeed,
\begin{align*}
    f(v+xp^{\lambda-j-1},r+yp^{\lambda-j-1}) +\tilde n
    &\equiv (2\Tilde{m}v-\vprod{\alpha,r}-\Tilde{\ell})xp^{\lambda-j-1}-\vprod{\alpha v+r-\beta,y}p^{\lambda-j-1} \\
    &\equiv 0 \pmod{p^{\lambda-1}}
\end{align*}
and 
\begin{align}
    2\Tilde{m}(v+xp^{\lambda-j-1})-\vprod{\alpha,r+yp^{\lambda-j-1}}-\Tilde{\ell} &\equiv 2\Tilde{m}v-\vprod{\alpha,r}-\Tilde{\ell}  \pmod{p^{j+1}}, \\
    \vprod{\alpha (v+xp^{\lambda-j-1})+r+yp^{\lambda-j-1}-\beta,z}&\equiv \vprod{\alpha v+r-\beta,z}  \pmod{p^{j+1}}
\end{align}
for all $z\in L$.

A similar argument shows that $(v+xp^{\lambda-j-1},r+yp^{\lambda-j-1})\in \mathcal M^*_j(p^\lambda)$ if and only if
\begin{equation}
f(v,r)+\tilde n+(2\Tilde{m}v-\vprod{\alpha,r}-\Tilde{\ell})xp^{\lambda-j-1}-\vprod{\alpha v+r-\beta,y}p^{\lambda-j-1} \equiv 0\pmod{p^\lambda},
\end{equation}
which holds if and only if
\begin{equation}\label{eq:in-M-lambda}
\frac{f(v,r)+\tilde n}{p^{\lambda-1}}+\frac{2\Tilde{m}v-\vprod{\alpha,r}-\Tilde{\ell}}{p^j}x - \frac{\vprod{\alpha v+r-\beta,y}}{p^j} \equiv 0\pmod{p}.
\end{equation}
We consider two cases.
\begin{enumerate}
    \item If 
    $2\Tilde{m}v-\vprod{\alpha,r}-\Tilde{\ell} \not\equiv 0\pmod{p^{j+1}}$, then $(2\tilde m v-\vprod{\alpha,r}-\tilde \ell)/p^j$ is invertible mod $p$, so for each $y$ the pair $(x,y)$ satisfies \eqref{eq:in-M-lambda} for $x$ in exactly one residue class mod $p$.
    \item Suppose that
    $2\Tilde{m}v-\vprod{\alpha,r}-\Tilde{\ell} \equiv 0\pmod{p^{j+1}}$.
    We identify $L$ with $\Z^g$ as explained in Section~\ref{sec:background}.
    Then, by \eqref{eq:new-xy-in-M-j}, there exists a basis $\{e_1,\ldots,e_g\}\subseteq \Z^g$ such that
    \[
        \vprod{\alpha v+r-\beta,e_i} \equiv 0 \pmod{p^j} \quad \text{ for } 1\leq i \leq g,
    \]
    but
    \[
        \vprod{\alpha v+r-\beta,e_1} \not\equiv 0 \pmod{p^{j+1}}.
    \]
    Write $y=\sum_i a_i e_i$, with $a_i\in \Z/p^{j+1}\Z$.
    Then \eqref{eq:in-M-lambda} holds if and only if
    \[
        a_1 \vprod{\alpha v+r-\beta,e_1}/p^j \equiv -\sum_{i=2}^g a_i \vprod{\alpha v+r-\beta,e_i}/p^j + (f(v,r)+\tilde n)/p^{\lambda-1} \pmod{p}.
    \]
    For each choice of $x$ and $a_2,\ldots,a_g$, the latter congruence holds for $a_1$ in exactly one residue class mod $p$.
\end{enumerate}
It follows that $M_j^*(p^\lambda) = \frac 1p M_j^*(p^{\lambda-1})$, and
this proves \eqref{eq:Mj-want}.
\end{proof}

\begin{lemma}
    Suppose that $p=2$ or $p\mid m$.
    If $N(p^\lambda)\ne p^gN(p^{\lambda -1})$ then \[\ell^2\equiv mn\pmod{2Np^{2\lfloor\frac{\lambda}{2}\rfloor}}.\]
\end{lemma}
\begin{proof}
Let $j=\lfloor\tfrac{\lambda}{2}\rfloor.$ 
By Lemma~\ref{lem:N=M-lambda/2}, $N(p^\lambda)=p^gN(p^{\lambda -1})$ unless 
\[
    M_j(p^\lambda) \ne \tfrac{1}{p} M_j(p^{\lambda-1}).
\] 
If this is the case, then $\mathcal M_j(p^{\lambda-1})$ is nonempty, so let $(v,r)\in \mathcal M_j(p^{\lambda-1})$.
Then we have 
\begin{align}
    f(v,r) + \tilde n &\equiv 0 \pmod{p^{\lambda-1}}, \label{eq:M1} \\
    2\Tilde{m}v-\vprod{\alpha,r}-\Tilde{\ell} &\equiv 0 \pmod{p^j}, \label{eq:M2} \\
    \vprod{\alpha v+r-\beta,y} & \equiv 0 \pmod{p^j} \quad \text{ for all } y\in L. \label{eq:M3}
\end{align}

Let $\gamma=\alpha v+r-\beta \in L'$ and $\hat \gamma = \Delta \gamma \in L$.
It will be convenient to rewrite the congruences \eqref{eq:M1}--\eqref{eq:M3} in terms of $\gamma$ and $\hat \gamma$.
First, a calculation shows that
\begin{equation}
    mv^2-2\ell v+n - \vprod{\hat \gamma,\gamma} = 2\Delta(f(v,r)+\tilde n),
\end{equation}
so by \eqref{eq:M1} we have
\begin{equation} \label{eq:M1'}
    m(mv^2-2\ell v+n) \equiv m\vprod{\hat \gamma,\gamma} \pmod{2Np^{\lambda}}.
\end{equation}
Here we have used that $p=2$ or $p\mid m$.
Second, if $\hat \alpha = \Delta\alpha \in L$ then we have
\begin{equation} \label{eq:mv-l}
    mv - \ell = \vprod{\hat \alpha, \gamma} + X,
\end{equation}
where, by \eqref{eq:M2},
\begin{equation} \label{eq:M2'}
    X = \Delta(2\tilde m v - \vprod{\alpha,r} - \tilde \ell) \equiv 0 \pmod{2Np^j}.
\end{equation}
Lastly, $p^{-j}\gamma\in L\otimes \Q$ so, by \eqref{eq:M3}, we have $\vprod{p^{-j}\gamma, y}\in \Z$ for all $y\in L$.
It follows from this and \eqref{eq:L'-def} that 
\begin{equation} \label{eq:M3'}
    \gamma_1 := p^{-j}\gamma \in L'.
\end{equation}

By \eqref{eq:M1'} and $\eqref{eq:mv-l}$ we have
\begin{align}
    mn-\ell^2 &= m(mv^2-2\ell v+n)-(mv-\ell)^2 \\
    &\equiv m\vprod{\hat \gamma,\gamma} - \vprod{\hat \alpha,\gamma}^2 - 2\vprod{\hat \alpha,\gamma} X \pmod{2Np^{2j}} \\
    &\equiv \vprod{\hat \alpha,\alpha}\vprod{\hat \gamma,\gamma} - \vprod{\hat \alpha,\gamma}^2 + 2\Delta \tilde m \vprod{\hat\gamma,\gamma} - 2 \vprod{\hat \alpha,\gamma} X \pmod{2Np^{2j}}.
\end{align}
We now use \eqref{eq:M3'} to obtain
\begin{align}
    \vprod{\hat\alpha,\alpha} \vprod{\hat\gamma,\gamma}-\vprod{\hat\alpha,\gamma}^2 
    &= \vprod{\hat\alpha, \vprod{\hat \gamma,\gamma}\alpha - \vprod{\hat\alpha,\gamma}\gamma} \\
    &= \Delta^2 p^{2j}\vprod{\alpha, \vprod{\gamma_1,\gamma_1}\alpha - \vprod{\alpha,\gamma_1}\gamma_1} \equiv 0 \pmod{2Np^{2j}}, \label{eq:aagg-agag}
\end{align}
where we used Lemma~\ref{lem:alpha-beta-gamma} in the second line to show that $\Delta(\vprod{\gamma_1,\gamma_1}\alpha - \vprod{\alpha,\gamma_1}\gamma_1) \in L$.
Using \eqref{eq:M3'} again we find that 
\begin{equation} \label{eq:gg}
\vprod{\hat\gamma,\gamma} = p^{2j}\vprod{\Delta\gamma_1,\gamma_1} \equiv 0\pmod{p^{2j}}.
\end{equation}
Finally, $\vprod{\hat\alpha,\gamma} = p^j\vprod{\hat\alpha,\gamma_1} \equiv 0\pmod{p^j}$.
The lemma follows.
\end{proof}

To finish the proof of Theorem~\ref{thm:exp-sum-general} we need to prove the upper bound for $\xi_{\alpha,\beta}(\ell,m,n,p^\lambda)$.
This will follow quickly from Lemma~\ref{lem:N=M-lambda/2} after we give an upper bound for $M_j(p^k)$.

\begin{lemma} \label{lem:M-bound}
    If $p^\nu\parallel 2\Delta$, $p^\mu\parallel m$, and $j\leq k\leq \lambda$ then 
    \[
    M_j(p^k) \leq 
    \begin{cases}
        p^{(g+1)(\lambda-j+\mu)+g\nu} & \text{ if } j\geq \nu+\mu, \\
        p^{(g+1)\lambda - j+\mu} & \text{ if } j \leq \nu+\mu-1.
    \end{cases}
    \]
\end{lemma}

\begin{remark*}
    Since $(-1)^{(g-1)/2}m$ is fundamental, we have $\mu=1$ for odd $p$ and $\mu\leq 3$ for $p=2$.
\end{remark*}

\begin{proof}
Suppose $(v,r)\in \mathcal M_j(p^k)$. 
If $(v+x,r+y)\in \mathcal M_j(p^k)$ as well, then
\begin{align*}
    2\Tilde{m}x&\equiv \vprod{\alpha, y} \pmod{p^j}, \\
    \vprod{\alpha,z}x&\equiv -\vprod{y,z} \pmod{p^j} \quad \text{ for all } z\in L.
\end{align*}
We apply the second congruence with $z=\hat\alpha := \Delta\alpha$ to get
\begin{equation}
    2\Delta\tilde m x \equiv \vprod{\hat\alpha,y} \equiv -\vprod{\hat\alpha,\alpha}x \pmod{p^j},
\end{equation}
i.e.~$mx\equiv 0 \pmod{p^j}$.
Thus $p^{j-\mu}\mid x$, so
\[
\vprod{\alpha,y}\equiv \vprod{y,z}\equiv 0 \pmod{p^{j-\mu}} \quad \text{ for all } z\in L.
\]
Let $y_1 = p^{-j+\mu}y\in L\otimes \Q$.
Then $\vprod{y_1,z}\in \Z$ for all $z\in L$, so by \eqref{eq:L'-def}, $y_1\in L'$.
It follows that $\Delta y_1\in L$, so $y\in p^{j-\mu-\nu}L \cap L$.

If $j\geq \mu+\nu$ then every element of $\mathcal M_j(p^k)$ is of the form
\begin{equation}
    (v+p^{j-\mu}x', r+p^{j-\mu-\nu}y')
\end{equation}
for some $x'\in \Z$, $y'\in L$.
Therefore $\mathcal M_j(p^k)$ has at most $p^{(g+1)\lambda-(j-\mu)-g(j-\mu-\nu)}$ elements. 
If $j<\mu+\nu$ then every element of $\mathcal M_j(p^k)$ is of the form
\begin{equation}
    (v+p^{j-\mu}x', r+y')
\end{equation}
for some $x'\in \Z$, $y'\in L$, so $M_j(p^k)\leq p^{(g+1)\lambda-(j-\mu)}$.
\end{proof}

By Lemma~\ref{lem:N=M-lambda/2} we have
\begin{equation}
    p^{-\lambda h}\left|N(p^\lambda) - p^g N(p^{\lambda-1})\right| \leq p^{-\lambda h}\max\left(M_{\lf}(p^\lambda), \tfrac 1p M_{\lf}(p^{\lambda-1})\right).
\end{equation}
If $\lf \geq \nu+\mu$ then Lemma~\ref{lem:M-bound} gives
\[
    p^{-\lambda h}\left|N(p^\lambda) - p^g N(p^{\lambda-1})\right| \leq p^{\nu g + \frac 72(g+1)}
\]
because $\mu\leq 3$.
On the other hand, if $\lf \leq \nu+\mu-1$ then $\lambda\leq 2\nu+5$, so Lemma~\ref{lem:M-bound} gives
\[
    p^{-\lambda h}\left|N(p^\lambda) - p^g N(p^{\lambda-1})\right| \leq p^{\frac{g+1}2\lambda - \lf + 3} \leq p^{\nu g+\frac 52g+\frac 72}.
\]
In either case we have $|\xi_{\alpha,\beta}(\ell,m,n,p^\lambda)|\leq p^{A\nu g}$ for some absolute constant $A$.
\qed

\printbibliography

\end{document}